\begin{filecontents*}{example.eps}
\documentclass[12pt]{article}
\usepackage{latexsym,amsmath,amsthm,verbatim,ifthen,amssymb}
\usepackage[all]{xy}

\newtheorem{theorem}{Theorem}[section]

\newtheorem{corollary}[theorem]{Corollary}
 
 \newtheorem{proposition}[theorem]{Proposition}
 \theoremstyle{definition}
 
 \theoremstyle{remark}
 \newtheorem{remark}[theorem]{Remark}
 
 \numberwithin{equation}{subsection}
\newcommand{\bz}{\mathbb Z}
\newcommand{\bq}{\mathbb Q}

\newcommand{\bak}{\operatorname{\text{BCH}}}

 \newcommand{\dlie}{\partial }

 \newcommand{\lib }{\mathbb{L}}

  \newcommand{\cyl }{\textrm{Cyl}\,}

 \newcommand{\ad }{\text{ad}}

 \newcommand{\cls }{\text{cls}}

 \newcommand{\map}{\text{map}}

    \newcommand{\lasu}{{\mathfrak{L}}}

\newcommand{\ees}{su}
\newcommand{\aes }{u}
\newcommand{\bes }{u'}
\newcommand{\equis }{x}
\newcommand{\ye }{y}

\newcommand{\yxbchl }{\bak_\ye( \ye, \equis )}

\begin{document}

\title{The Lawrence-Sullivan construction is the right model of $I^+$}


\author{Urtzi Buijs\footnote{Partially supported by the {\em MICIN}  grant MTM2010-15831 and by the Junta de
Andaluc\'\i a grant FQM-213} $$  and Aniceto Murillo\footnote{Partially supported by the {\em MICIN}  grant MTM2010-18089 and by the Junta de
Andaluc\'\i a grants FQM-213 and P07-FQM-2863. \vskip 1pt 2000 Mathematics Subject
Classification: 55P62, 11B68.\vskip
 1pt
 Key words and phrases: Lawrence-Sullivan model, DGL cylinder, DGA cylinder,
rational homotopy theory, Euler identity.}}



\maketitle

\begin{abstract}
We prove that the universal enveloping algebra of the Lawrence-Sullivan construction is a particular perturbation of the complete Baues-Lemaire cylinder of $S^0$. Together with other evidences we present, this exhibits the Lawrence-Sullivan construction as the right  model of $I^+$.  From this, we also deduce a generalized Euler formula on Bernoulli numbers.
\end{abstract}

\section{Introduction}\label{intro}
In \cite{lasu} R. Lawrence and D. Sullivan constructed what they called a {\em free Lie model for the interval}. It is a complete free differential Lie algebra $\lasu=(\widehat\lib(a,b,z),\partial)$ in which $a$ and $b$ are {\em Maurer-Cartan} or {\em flat} elements and the differential on $z$ is the only one which makes $a$ and $b$ gauge equivalent (see next section for a precise definition). In very general terms, it is based on the idea that $0$-simplices (points) are represented by Maurer-Cartan elements and $1$-simplices (paths) correspond to {\em gauge transformations} between the endpoints.

  In this beautiful construction and the deep ideas behind it there is, however, a particular Maurer-Cartan element which has not received attention, namely $0\in \lasu_{-1}$. Taking that into account we find three Maurer-Cartan elements $0,a,b\in \lasu_{-1}$, and thus three points in the space that it represents, together with a gauge transformation from $a$ to $b$ representing a path joining two of the three vertices. Hence, what $\lasu$ really models is $I^+=I\stackrel{\cdot}{\cup} \{*\}$, the disjoint union of the interval and an exterior point.

  That the spatial realization of this complete Lie algebra is in fact of the homotopy type of two contractible components is a particular instance of the more general context of homotopy theory in the category of (unbounded) differential graded Lie algebras, or more generally, $L_\infty$ algebras \cite{bumu}.

  In this paper we give to the Lawrence-Sullivan construction the consideration that it deserves by proving that it satisfies essential functorial properties which reflect the role of $I^+$ in the based homotopy category. We begin by observing (see Proposition \ref{propo}) that the differential in $\lasu$ is characterized by being the only one for which the following holds:  two Maurer-Cartan elements $x,y\in L_{-1}$ of a given differential graded Lie algebra $L$, DGL henceforth, are gauge equivalent if and only if there exist a morphism of differential graded algebras $ \lasu\to L $ sending $a$ to $x$ and $b$ to $y$. This readily implies, as we will remark, that the classical Quillen notion of homotopy of morphisms corresponds, under this scope,  to the standard homotopy notion in the based category in which $S^0\wedge I=I^+$ is the right cylinder. In other words, $\lasu$ is a good cylinder of the model of $S^0$ given by $(\lib(u),\partial)$, $\partial(u)=-\frac{1}{2}[u,u]$.

   Our  main result, Theorem \ref{teorema}, corroborates this assertion. We show that the universal enveloping algebra of $\lasu$ is a complete tensor algebra whose differential is a perturbation of the classical Baues-Lemaire cylinder \cite{bale} of the tensor model of $S^0$ given by $T(u)$, with $u$ of degree $-1$ and  $d(u)=-u\otimes u$.

    With all of the above a model of $I$ is then $(\lib(a,z),\partial)$,  obtained as the quotient of $\lasu$ by $(\lib(b),\partial)$. Theorem \ref{teorema2} illustrates this situation.

    We finish with an application of Theorem \ref{teorema} in number theory by obtaining, in the same spirit as in \cite{patan}, a  generalized {\em Euler formula} on Bernoulli numbers.

 \section{Notation and tools}

In this note, the coefficient field $\mathbb K$ is assumed to be of characteristic zero. Differential Lie algebras, DGL's from now on, as well as  any other graded object, are considered $\bz$-graded. Given a graded vector space $V$, $\lib(V)$ denotes the {\em free lie algebra} generated by $V$. If in the tensor algebra $T(V)=\sum_{n\ge0}T^n(V)$ we consider the Lie structure given by commutators, $\lib(V)$ it is the Lie subalgebra generated by $V$. In the same way, replacing $T(V)$ by the {\em complete tensor algebra} $\widehat{T}(V)=\Pi_{n\ge0}T^n(V)$, we obtain $\widehat\lib(V)$, the {\em complete free Lie algebra} generated by $V$. A generic element of $\widehat T(V)$ will be written as a formal series $\sum_{n\ge 0}\phi_n$, $\phi_n\in T^n(V)$. Note that $T(V)\subset \widehat T(V)$ and $\lib(V)\subset\widehat\lib(V)$. The universal enveloping algebra $U\lib(V)$ of  $\lib(V)$ extends to the complete  free Lie algebra to produce a graded algebra $U\widehat \lib(V)$ naturally isomorphic to $\widehat T(V)$.

 A {\em Maurer-Cartan} or {\em flat} element of a given DGL is a  degree $-1$ element $a$ which satisfies $\partial(a)=-\frac{1}{2}[a,a]$. We denote by $MC(L)$  the set of all Maurer-Cartan elements of  $L$.

Let  $L$ be either a complete free Lie algebra or any DGL in which the adjoint action of  $L_0$ is locally nilpotent, i.e., for any $x\in L_0$ there is an integer $i$ such that $\ad_x^i=0$. The {\em gauge action} of $L_0$ on $MC(L)$  \cite{lasu,mane} is defined as follows: given $x\in L_0$ and $a\in MC(L)$,
$$x*a=e^{\ad_x}(a)-\frac{e^{\ad_x}(\partial x)-(\partial x)}{\ad_x(\partial x)}=e^{\ad_x}(a)-f_x(\partial x),$$
where $e^{\ad_x}=\sum_{n\ge0}\frac{(\ad_x)^n}{n!}$ and, as operator,
$$f_x=\frac{e^{\ad_x}-id}{\ad_x}.
$$
Explicitly,
$$
x*a=\sum_{i\ge0}\frac{\ad_x^i(a)}{i!}+\sum_{i\ge0}\frac{\ad_x^{i}(\partial x)}{(i+1)!}.
$$
This can also be introduced in this way: consider $L\otimes  \Lambda t=L[t]$ the DGL   in which $t$ has degree zero, the Lie bracket is given by the one in $L$ and multiplication on $\mathbb K[t]$, and the differential arises from that on $L$ and setting $\partial t=0$. Then, $x*a=p(1)$ where $p(t)\in  L[t]$ is the only formal power series with coefficients in $L$ which is solution of the differential equation:
$$
\begin{aligned}
&p'(t)=\partial x -\ad_xp(t),\\
&p(0)=a.
\end{aligned}
$$
Indeed, if we write $p(t)=\sum_{i\ge0}a_it^i$, $a_i\in L$, the only solution for this equation is given recursively by
$$
a_0=a,\qquad a_1=\partial x -\ad_x(a),\qquad a_{n+1}=-\frac{\ad_x(a_n)}{n+1}.
$$
That is,
$$
a_n=\frac{(-\ad_x)^{n-1}}{n!}(\partial x)+\frac{(-\ad_x)^{n}}{n!}(a),
$$
and therefore $p(1)=e^{\ad_x}(a)-f_x(\partial x)$

The gauge action determines an equivalence relation among flat elements which coincide with the usual homotopy relation  on $MC(L)$ (see for instance \cite[Theorem 5.5]{mane}) that we now recall. Two elements $u,v\in MC(L)$ are said to be homotopic, and we write $u\sim v$, if there is a Maurer-Cartan element $\phi\in L\otimes \Lambda(t,dt)$ such that $\varepsilon_0(\phi)=u$ and $\varepsilon_1(\phi)=v$. Here, $\Lambda(t,dt)$ is the free commutative algebra generated by $t$ and $dt$, of degrees $0$ and $1$ respectively, and $\varepsilon_0,\varepsilon_1\colon  L\otimes \Lambda(t,dt)\to L$ are the DGL morphisms obtaining evaluating $t$ at $0$ and $1$ respectively.

The {\em Lawrence-Sullivan construction}, denoted by $\lasu$, is the complete free DGL $(\widehat\lib(a,b,z),\partial)$ in which $a$ and $b$ are flat elements and
$$
\partial(z)=[z,b]+\sum_{i=0}^\infty\frac{B_i}{i!}\ad_z^i(b-a),
 $$
 where $B_i$ denotes the $i^{\text{th}}$ Bernoulli number. Equivalently, as shown in \cite{lasu},
 $$\partial z=\ad_z(b)+f_z^{-1}(b-a),$$
 where, again as operator,
 $$
 f_z^{-1}=\frac{\ad_z}{e^{\ad_z}-id}.
 $$
  Observe that $f_z^{-1}$ is indeed the inverse operator of $f_z$. Another inductive description of the differential in this complete free Lie algebra was suggested in \cite{lasu} and shown to be equivalent to the above in \cite[Main Theorem]{patan}.

  Last result of next section assumes basic knowledge on $L_\infty$ algebras. From the rational homotopy theory point of view on these objects, we refer, for instance, to \cite{ge}.

\section{The DGL cylinder of  $S^0$ and a model of the interval}

We begin by the following observation:

\begin{proposition}\label{propo}
Two Maurer-Cartan elements $u,v\in L_{-1}$ are homotopic if and only if there is a DGL morphism $\Phi\colon\lasu\to L$ such that $\Phi(a)=u$ and $\Phi(b)=v$.
\end{proposition}
\begin{proof}

Let
$\Phi \colon \lasu\to L$
such that $\Phi (a)=u$ and $\Phi (b)=v$, with $u,v\in MC(L)$, and  consider $w=\Phi
(z)\in L_0$. Then
\begin{eqnarray*}
\partial w &=&\Phi (\partial z)=\Phi \bigl(\ad_{z}(b)+\frac{\ad_{z}}{e^{\ad_{z}}-id}(b-a)\bigr)\\
&=&\ad_{\Phi (z)}\Phi (b)+\frac{\ad_{\Phi (z)}}{e^{\ad_{\Phi
(z)}}-id}\bigl(\Phi (b)-\Phi (a)\bigr)=\ad_w(v)+\frac{ \ad_{w}}{e^{\ad_w }-id}(v-u).
\end{eqnarray*}
Then,
$$
 \frac{e^{\ad_w}-id}{ \ad_{w}}(\partial
w)=\frac{e^{\ad_w }-id}{ \ad_{w}}(\ad_wb)+v-u\\
=e^{\ad_w}(v)-u,
$$
and thus, $u=e^{\ad_w}(v)-f_w (\partial w)$. In other words,
 $u$ and $v$ are gauge equivalent, and thus, homotopy equivalent.

Reciprocally if $u,v\in MC(L)$ are homotopic, there is  $w\in L_0$ such that
$u=e^{\ad_w}(v)-f_w (\partial w)$.
define $\Phi (b)=v$, $\Phi (a)=u$ and $\Phi (z)=w$ and the
above computation shows that this is a DGL morphism.
\end{proof}

\begin{remark} We see that $\lasu$ fits perfectly in the classical Quillen approach to rational homotopy \cite{qui} by which the homotopy category of reduced DGL's is equivalent to the based homotopy category of simply connected rational spaces. In this category, two pointed maps $f,g\colon (X,x_0)\to (Y,y_0)$ are homotopic via a based homotopy $(X\times I,x_0\times I)\to (Y,y_0)$. But this corresponds, via the exponential, to a based map $S^0\wedge I=I^+\to \map^*(X,Y)$ sending the exterior point to the constant map. In the algebraic setting, let $C$ and $L$ be a  differential graded coalgebra and a reduced DGL, models of $X$ and $Y$ respectively, with the appropriate connectivity restrictions, and  let $\varphi,\psi\colon \mathcal L(C)\to L$ be models of $f$ and $g$. Here $\mathcal L$ denotes the Classical Quillen functor, see for instance \cite[\S22]{fehatho}. It is known \cite[Theorem 10]{bufemu} that the convolution DGL $Hom(C_+,L)$ is a DGL model of the non connected pointed mapping space $\map^*(X,Y)$. Moreover, the restrictions $\varphi_{|_{C_+}},\psi_{|_{C_+}}\colon C_+\to L$ are Maurer-Cartan elements of this DGL. By the proposition above, these morphisms are homotopic if and only if there exists $\Phi\colon \lasu\to Hom(C_+,L)$ such that $\Phi(a)=\varphi_{|_{C_+}}$ and $\Phi(b)=\psi_{|_{C_+}}$.
\end{remark}

In \cite[\S1]{bale}, H. Baues and J.M. Lemaire defined a cylinder of a free differential graded algebra, DGA henceforth, $(T(V),d)$. It is
the free tensor algebra $\cyl T(V)=T(V\oplus V'\oplus sV)$ in which $V'$ is a copy of $V$ and $sV$ denotes the suspension of $V$, $(sV)_p=V_{p-1}$. The differential in $V$ and $V'$ are defined so that the inclusions $i_0, i_1 \colon T(V) \hookrightarrow T(V\oplus V'\oplus sV)$, where $i_0(v)=v$, $i_1(v)=v'$, with $v\in V$,
are DGA morphisms. Consider in $\cyl T(V)$ the  $(i_0,i_1)$-derivation $S$ of degree $-1$    given by $Sv=Sv'=sv$ and $Ssv=0$. Then, the differential in $sV$ is defined by $dsv=v-v'-Sdv$, $v\in V$.

This is a right cylinder in the category of DGA's as both $i_0,i_1$ are injective and the projection $p\colon T(V\oplus V'\oplus sV)\to T(V)$, $p(v)=p(v')=v$, $p(sv)=0$, $v\in V$, is a quasi-isomorphism for which $pi_0=pi_1=1_{|_{T(V)}}$.

When $V=\langle u\rangle$ with $u$ of degree $-1$ and $du=-u\otimes u$ we obtain $\cyl T(u) =(T(u\oplus u'\oplus su),d)$ with differential $du=-u\otimes u$, $du'=-u'\otimes u'$ and $dsu=u'-u+su\otimes u'-u\otimes su$.
 Note that $T(u)$ can be thought of as a model for $S^0$ as  $\cyl T(u) $ can be used  as a cylinder to describe the usual notion of homotopy in the category of DGA's \cite{ma}. In $\widehat T(u\oplus u'\oplus su)$ we define $D$ as the only derivation which extends $d$ in $u$ and $u'$ and,
 $$
Dsu= su\otimes u'-u'\otimes su+ \sum_{n\ge 0} \sum_{p+q=n}(-1)^q\frac{B_{n}}{p!q!}su^{\otimes p}\otimes(u'-u)\otimes su^{\otimes q}\eqno{(1)}
$$
where each $B_{n}$ denotes the $n^{th}$ Bernoulli number.
 Denote by $\widehat\cyl T(u) $ the pair $(\widehat T(u\oplus u'\oplus su),D)$. We now show that $D$ is indeed a differential for which:

\begin{theorem}\label{teorema}   $U\lasu$ is isomorphic as DGA's to $\widehat\cyl T(u)$.
\end{theorem}

\begin{proof}We show that the natural isomorphism of graded algebras $U\lasu\cong \widehat\cyl T(u) $  commutes with differentials on generators. This also guarantees that $D$ is indeed a differential. Write $Du=du=-u\otimes u= -\frac{1}{2}u\otimes u - \frac{1}{2}u\otimes u $ which only arises from $-\frac{1}{2}[a,a]=\partial a$ via the injective map $\lasu\hookrightarrow U\lasu\cong \widehat\cyl T(u) $ provided by the Poincar\'e-Birkhoff-Witt theorem. The same applies to $Du'$ and $\partial b$.

On the other hand, if we denote by $\cls$ the class of the corresponding element of $U\lasu=\widehat T(\lasu)/\sim$, we have:

$$
\begin{aligned}
\cls\{Dsu\}&=\cls\{su\otimes u'-u'\otimes su+ \sum_{n\ge 0} \sum_{p+q=n}(-1)^q\frac{B_{n}}{p!q!}su^{\otimes p}\otimes(u'-u)\otimes su^{\otimes q}\}=\\
&=\cls\{\ad_{su}u'+\sum_{n\ge 0} \sum_{p+q=n}(-1)^q\frac{B_{n}}{n!}\binom{n}{q}su^{\otimes p}\otimes(u'-u)\otimes su^{\otimes q}\}=\\
&=\cls\{\ad_{su}u'+\sum_{n=0}^\infty \frac{B_n}{n!}\sum_{k=0}^n(-1)^k\binom{n}{k}su^{\otimes n-k}\otimes{(u'-u)}\otimes su ^{\otimes k}\}=\\
&=\cls\{\ad_{su}u'+\sum_{n=0}^\infty \frac{B_n}{n!}[\underbrace{su ,[su ,\dots ,[su }_{n},{u'-u} ]]\dots ]\}=\\
&=\cls\{\ad_{su}u'+\sum_{n=0}^\infty \frac{B_n}{n!}\text{ad}_{su}^n({u'-u} )\}
\end{aligned}
$$

Again, by the Poincar\'e-Birhoff-Witt theorem, this class only arises from $\partial z$ in the injection $\lasu\hookrightarrow U\lasu\cong \widehat\cyl T(u) $ and the theorem is proved.

\end{proof}

Recall that, in a complete tensor algebra $\widehat T(V)$, the {\em Baker-Campbell-Hausdorff formula} of $x,y\in V$ reads,
 $$
{\bak}(y,x)= \sum_{k=1}^\infty \frac{(-1)^{k-1}}{k}\Bigl(
\sum_{\stackrel{p,q=0}{p+q>0}}^\infty
\frac{y^px^q}{p!q!}\Bigr)^k=\sum_{k=1}^\infty\frac{(-1)^{k-1}}{k} \sum
\frac{y^{p_1}x^{q_1}\cdots y^{p_k}x^{q_k}}{p_1!q_1!\cdots
p_k!q_k!}
$$
where in the last term, the internal sum is taken over all possible
collections $(p_1, \dots, p_k, q_1,\dots ,q_k)$ of integral
non-negative numbers such that
$p_1+q_1>0, \dots , p_k+q_k>0$. Here, and in what follows, for simplicity in the notation, we have omitted $\otimes$ in  all tensor products. As in the classical case, this formula is obtained as
$$
{\bak}(y,x)=\log(e^ye^x)$$
where
$$e^x=\sum_{n\ge 0}\frac{x^n}{n!}$$ and, for a given $\varphi\in\widehat T^+(V)=\Pi_{i\ge 1}T^i(V)$,
$$
\log(1+\varphi)=\sum_{n\ge0}(-1)^{n-1}\frac{\varphi^n}{n}.
$$
We denote by ${\bak}_y(y,x)$ the part of ${\bak}(y,x)$ obtained by considering only those summands which are linear on $y$. Then, it is known, see for instance \cite[Lecture 6]{post} for the ungraded case, that in $U\widehat\lib(V)\cong \widehat T(V)$,
 $$
\cls\{{\bak}_{y}(y,x)\}=\cls\{\sum_{i\ge 0}\frac{B_i}{i!}\ad^i_{x}y\}.
$$
Hence, from Theorem \ref{teorema} we immediately obtain:

\begin{corollary}\label{BCH} With the notation in the theorem above,
$$\bak_{u'-u}(u'-u,su)=\sum_{p+q=n}(-1)^q\frac{B_{p+q}}{p!q!}su^{\otimes p}\otimes(u'-u)\otimes su^{\otimes q}$$
and the formula for $Dsu$ in equation $(1)$ can be rewritten as,
$$
Dsu=su\otimes u'-u'\otimes su+\bak_{u'-u}(u'-u,su).
$$
\hfill$\square$
\end{corollary}

\begin{remark} (1) Observe that  the projection $p\colon\widehat\cyl(u) \to (T(u),d)$, $p(u)=p(u')=u$, $p(su)=0$,  is again a quasi-isomorphism as it is so on the indecomposables. Thus, $\widehat\cyl (u)$ is again a cylinder for $(T(u),d)$.

(2) Consider the DGL $(\lib(u),\partial)$, $\partial u=-\frac{1}{2}[u,u]$, and observe that  the universal enveloping algebra functor on the inclusions $(\lib(u),\partial)\hookrightarrow \lasu$ mapping $u$ to $a$ and $b$ respectively, provide, by the theorem above, the natural inclusions $i_0,i_1\colon (T(u),d)\hookrightarrow \widehat\cyl T(u)$. That is, the Lawrence-Sullivan construction  is a good DGL cylinder of $(\lib(u),\partial)$. Now observe that this is  a DGL model of $S^0$. Indeed,  its commutative cochain  graded algebra  \cite[\S23]{fehatho}, $C^*(\lib(u),\partial)$, can be easily computed to yield
 $(\Lambda(x,y),d)$ with $x$ and $y$ of degree $0$ and $-1$ respectively,  $dx=0$ and $dy=\frac{1}{2}(x^2-x)$. This is a model of $S^0$ under any possible interpretation. On the one hand, it is an easy exercise to show that the geometric simplicial  realization of this algebra has the homotopy type of $S^0$,  each of its points given by the augmentations sending $x$ to $0$ and $1$ respectively. On the other hand, consider the DGA model of $S^0$ given  by $\bq \alpha\oplus\bq\beta$ with $\alpha$ and $\beta$ of degree $0$, and products $\alpha^2=\alpha$, $\beta^2=\beta$ and $\alpha\beta=0$.  Note that the identity in this algebra is $\alpha+\beta$. Hence, replacing $\alpha+\beta$ by $1$ and $\beta$ by $x$,  this DGA is isomorphic to $\bq\oplus\bq x$ with $x^2=x$, which is quasi-isomorphic to $(\Lambda (x,y),d)$ with $dy=\frac{1}{2}(x^2-x)$.

 In other words, $\lasu$  geometrically describes the cylinder of $S^0$ in the pointed category, namely, $I^+$.

 (3) In the same way, a good candidate for a complete model of the interval $I$ would be then a model  of the cofibre of the based map $S^0\to I\dot\cup *$ which sends the non base point of $S^0$ to any of the endpoints of the interval. In DGL's this would be modeled by the inclusion $
   (\lib(b),\partial)\hookrightarrow \lasu
   $
  whose cofibre is
$$
\lasu_I=(\widehat{\mathbb{L}}(a,z),\partial)\quad\mbox{in which}\quad\dlie(a)=-\frac{1}{2}[a,a] \quad\mbox{and}\quad \dlie(z)=-\sum_{i\ge 0}\frac{B_i}{i!}{\rm ad}^i_z(a).
$$
\end{remark}

Again, the fact that the realization of $\lasu_I$ is a contractible space is a particular case of a much broader setting \cite{bumu}. Here, we corroborate this fact with a variant of \cite[Theorem 4.1]{bugumu} that we now present.

Let $L$ be any DGL and consider, for each $\lambda\in \mathbb K$, the DGL morphism
$$
\varepsilon_\lambda\colon L\otimes \Lambda (t,dt)\to L,\qquad \varepsilon(p(t)+q(t)dt)=p(\lambda)
 $$
 given by evaluating in $\lambda$. Here $p,q\in L\otimes \Lambda t=L[t]$ are considered, as before, polynomials in $t$ with coefficients in $L$. If we restrict to polynomials with no constant terms and evaluate in $1$ the resulting DGL morphism,
$$
\varepsilon_1\colon L\otimes \Lambda^+(t,dt)\longrightarrow L
$$
is a DGL model in the classical Quillen sense
of the evaluation fibration
$$
ev\colon \map^*(I,Y)\longrightarrow Y,\qquad ev(\gamma)=\gamma(1),
$$
whenever $L$ is a reduced DGL model of $Y$ \cite{bufemu}. Recall that $\map^*$ denotes the pointed mapping space. We will not assume here any bounding condition on $L$ but will keep in mind this geometrical interpretation.

On the other hand, consider the chain map
$$
\varepsilon_a\colon s^{-1}Der(\lasu_I,L)\to L,\qquad \varepsilon_a(\theta)=\theta(a),
$$
in which $s^{-1}Der(\lasu_I,L)$ are the {\em desuspended} derivations with the usual differential. This chain complex admits a DGL structure only {\em up to homotopy}, i.e., an $L_\infty$ structure  in which the higher brackets $\{\ell_k\}_{k\ge 2}$ are given as follows (see \cite{bu,bugumu}):
$$
\begin{aligned}
&s\ell_k(s^{-1}\theta_1,\ldots, s^{-1}\theta_k)(z)=\\
&(-1)^{\epsilon}\frac{B_{k-1}}{(k-1)!}\sum_{\sigma\in S_k}\epsilon' [\theta_{\sigma(1)}z,[\theta_{\sigma(2)}z,\ldots [\theta_{\sigma(k-1)}z,\theta_{\sigma(k)}a]]\ldots].
\end{aligned}
$$
where $\epsilon=k+\sum_{j=1}^{k}(k+1-j)|\theta_j|$, and $\epsilon'$ is given by the Koszul convention. Then, we have:

\begin{theorem}\label{teorema2}
the map
$$
\varphi\colon s^{-1}Der(\lasu_I,L)\longrightarrow L\otimes \Lambda^+(t,dt),\qquad \varphi(s^{-1}\theta)=\theta(a)\otimes t+\theta(z)\otimes dt
$$
is a quasi-isomorphism of $L_\infty$ algebras which makes commutative the diagram,

$$
\xymatrix{
 s^{-1}Der(\lasu_I,L)\ar[rd]_{\varepsilon_a}\ar[rr]^{\varphi}_{\simeq}&& L\otimes \Lambda^+(t,dt)\ar[ld]^{\varepsilon_1}\\
&L&
}
$$

\end{theorem}

\begin{proof}
Observe that $\varphi$ is precisely the quasi-isomorphism of $L_\infty$ algebras denoted by $Q$ in the proof of \cite[Theorem 4.1]{bugumu}. We have just extend it to non bounded derivations and replace $s^{-1}Der(\mathcal L(\Lambda (t,dt))^\sharp,L)$ by the isomorphic DGL $L\otimes\Lambda(t,dt)$. On the other hand, the diagram above trivially commutes.
\end{proof}

\section{A generalized Euler identity}

Recall that
 Bernoulli numbers can be recursively defined  in several ways starting with $B_0=1$. One is via the  identity
$$-nB_n=\sum_{k=1}^n \binom{n}{k}B_kB_{n-k}+nB_{n-1}$$ which becomes the {\em Euler formula}
$$-\frac{(n+1)B_n}{n!}=\sum_{k=2}^{n-2} \frac{B_k}{k!}\frac{B_{n-k}}{(n-k)!}$$
when $n$ is  an even integer greater than $2$.  Here we prove an extended version which may be compared, for instance, to Miki \cite{mi} or Matiyasevich \cite{mati} identities:

\begin{theorem}

For any even $n\ge2$ and $0\le m\le n-1$, the following holds:
$$
 -\frac{B_{n}}{n!}\binom{n+1}{n-m}=\sum_{i=2}^{m}\frac{B_i}{i!}\frac{B_{n-i}}{(n-i)!}\binom{n-i}{n-m-1}
-\sum_{j=2}^{n-m-1}\frac{B_j}{j!}\frac{B_{n-j}}{(n-j)!}\binom{n-j}{m}.
$$
\end{theorem}
Here, whenever $m$ or $n-m-1$ are smaller than $2$, the corresponding summand in the formula above is considered to be zero.
Note also that, for  $m=n-1$ we recover Euler's formula. In \cite[Proposition 8]{patan} a different and interesting formula involving Bernoulli number is deduced from the inductive definition of the differential in the Lawrence-Sullivan construction.

\begin{proof} In $\widehat \cyl T(u)$, from now on and to avoid excessive notation, we often omit the symbol $\otimes $. For the same purpose we redefine  $\ye=\bes -\aes $ and $\equis
=\ees$.

Using the formula for $Dx$ in Corollary \ref{BCH},  a short computation shows that $D^2x=0$ translates to:
$$
D^2x=D\bak_{y}(y,x)+\bak_{y}(y,x)\otimes u'+ u'\otimes \bak_{y}(y,x)\eqno{(2)}
$$
in which
$$\yxbchl=\sum_{n=0}^\infty \sum_{p+q=n}c_{(p,q)}\equis^p\ye \equis^q$$
where
$$c_{(p,q)}=(-1)^{q}\frac{B_{p+q}}{(p+q)!}\binom{p+q}{q}.\eqno{(3)}$$

As $D \equis^m= \equis ^m \bes -\bes \equis^m+\sum_{i=0}^{m-1}
\equis^i\yxbchl \equis ^{m-1-i},$ a straightforward computations shows that
$$
D(\equis ^p\ye \equis^q)=-\bes \equis^p\ye
\equis^q-\equis^p\ye \equis^q\bes + \Gamma
$$
with
$$
\Gamma=\equis^p\ye^2\equis^q+\sum_{i=0}^{p-1}\equis^i \yxbchl \equis^{p-i-1}  \ye \equis^q- \sum_{i=0}^{q-1}\equis^p\ye\equis^i \yxbchl \equis^{q-i-1}.
$$
Hence,
$$
\begin{aligned}
&D\,\yxbchl=\sum_{n=0}^\infty \sum_{p+q=n}c_{(p,q)}D(\equis^p\ye\equis^q)=\\
&=-\bes \otimes \sum_{n=0}^\infty \sum_{p+q=n}c_{(p,q)}\equis^p\ye \equis^q -\sum_{n=0}^\infty  \sum_{p+q=n}c_{(p,q)}\equis^p\ye \equis^q\otimes \bes +\sum_{n=0}^\infty \sum_{p+q=n}c_{(p,q)}\Gamma =\\
&=-\yxbchl \otimes \bes -\bes \otimes \yxbchl  +\sum_{n=0}^\infty \sum_{p+q=n}c_{(p,q)}\Gamma.
\end{aligned}
$$
Therefore, at the sight of equation $(2)$,
$$
\sum_{n=0}^\infty
\sum_{p+q=n}c_{(p,q)}\Gamma=0.
$$
In particular, for each $p$ and $q$, the coefficient of $x^py^2x^q$ in this formula is zero. Another straightforward computation shows that this coefficient is:
$$c_{(p,q)}+\sum_{i=0}^{p}
c_{(p+1-i,q)}c_{(i,0)}-\sum_{j=0}^{q}c_{(p,q+1-j)}c_{(0,j)}=0.\eqno{(4)}
$$
Note that, in this expression, the term for $i=j=1$ cancels with $c_{(p,q)}$ since $c_{(1,0)}=-c_{(0,1)}=-\frac{1}{2}$. On the other hand, the term for $i=j=0$ is $c_{(p+1,q)}-c_{(p,q+1)}$ as $c_{(0,0)}=1$.
Thus, equation $(4)$ reads,
$$c_{(p,q+1)}-c_{(p+1,q)}=\sum_{i=2}^{p}
c_{(p+1-i,q)}c_{(i,0)}-\sum_{j=2}^{q}c_{(p,q+1-j)}c_{(0,j)}.$$
Finally, replace each of these numbers by its value of formula $(3)$ and the theorem follows.
\end{proof}


\begin{thebibliography}{99}

\bibitem{bale} H. J. Baues and J. M. Lemaire, Minimal models in homotopy theory, {\em Math. Ann.}, {\bf 225} (1977), 219--242.

\bibitem{bu} U. Buijs, An explicit $L_\infty$ structure for the components of mapping spaces,  {\em Topology and its Applications} {\bf 159} (3) (2012),  721--732.

\bibitem{bufemu} U. Buijs, Y. F\' elix and A. Murillo, Lie models
for the components of sections of a nilpotent fibration, {\em
Trans. Amer. Math. Soc. } {\bf 361}(10) (2009), 5601--5614.

\bibitem{bugumu} U. Buijs, J. Guti\'errez and A. Murillo, Derivations, the Lawrence-Sullivan interval and the Fiorenza-Manetti mapping cone, preprint, 2011

\bibitem{bumu} U. Buijs  and A. Murillo, Algebraic models of non connected spaces and homotopy theory of $L_\infty$ algebras, preprint 2012.

\bibitem{fehatho} Y. F\'elix, S. Halperin, J. Thomas, {\em Rational
homotopy theory}, {\em Springer GTM }, {\bf 205 } , 2000.

\bibitem{ge} E. Getzler, Lie theory for nilpotent $L_\infty$ algebras. \textit{Ann. of Math (2)} {\bf 170} (2009), (1), 271--301.



\bibitem{lasu} R. Lawrence and D. Sullivan, A free diferential Lie algebra for the interval, preprint 2010.



\bibitem{ma} M. Majewski, Rational homotopical models and
uniqueness, {\em Mem. Amer. Math. Soc.}, {\bf 682} (2000).

\bibitem{mane} M. Manetti, Deformation theory via differential graded Lie algebras, Seminari di Geometria Algebrica 1998-1999, Scuola Normale Superiore (1999).

\bibitem{mati}
     Y. Matiyasevich, Identities with Bernoulli numbers,\hfill\break {\em
http://logic.pdmi.ras.ru/yumat/Journal/Bernoulli/bernulli.htm}, 1997.


\bibitem{mi} H. Miki, A relation between Bernoulli numbers, {\em J. Number Theory} {\bf 10} (1978),
     297--302.

\bibitem{patan}
P. E. Parent, D. Tanr\'e, Lawrence-Sullivan models for the interval,  {\em Topology and its Applications} {\bf 159} (1) (2012),  371--378.

\bibitem{post} M. M. Postnikov, {\em Lie groups and Lie algebras}, {\em Nauka}, Moscow, 1982.

\bibitem{qui} D. Quillen, Rational homotopy theory, {\em Ann.
of Math.}, {\bf 90} (1969), 205--295.






\end{thebibliography}
\end{document}